\documentclass[10pt,a4paper]{article}
\usepackage[latin1]{inputenc}
\usepackage[margin=3cm]{geometry}
\usepackage{amsmath}
\usepackage{amsfonts}
\usepackage{stmaryrd}
\usepackage{amssymb}
\usepackage{multirow}
\usepackage{amsthm}
\usepackage{enumerate}
\usepackage[table]{xcolor}
\usepackage{hhline}
\usepackage{makecell}
\usepackage{tikz}
\usepackage{ifthen}
\usepackage{subfiles}
\usepackage{accents}
\usepackage{mathrsfs}
\newenvironment{tproof}{\begin{proof}[\textbf{Proof}]}{\end{proof}}
\newlength{\dhatheight}

\newcommand{\Aut}[1]{\mathrm{Aut}_{#1}}

\newcommand{\Lie}[1]{\mathrm{Lie}(#1)}

\newcommand{\mult}[2]{\mathrm{mult}(#1,#2)}

\newcommand{\exchange}[2]{\mathrm{ue}\left(#1,#2\right)}

\newcommand{\kkk}{\mathfrak{k}}

\newcommand{\AAA}[1][]{\mathcal{A}_{\kkk#1}}

\newcommand{\BBnk}[1][]{\mathcal{B}_{#1}}

\newcommand{\UUU}[1]{\mathcal{N}_{#1}}

\newcommand{\s}{\setminus}

\newcommand{\Pri}[1][]{\mathcal{C}_{\kkk#1}}

\newcommand{\QQ}{\mathrm{Q}}

\newcommand{\F}{\mathrm{F}}

\newcommand{\Z}{\mathrm{Z}}

\newcommand{\Om}{\mathcal{O}}

\newcommand{\Omm}[1]{\mathcal{O}_{\mathbf{f}}(#1)}

\newcommand{\A}{\mathbb{A}_{\kkk}}

\newcommand{\e}{\text{e}}
\newcommand{\di}{6}

\newcommand{\ci}[4][\di]{\ifthenelse{\equal{#4}{a}}{\draw (#2-0.5,#1+0.5-#3)[gray!70,fill=gray!70]circle(0.17);}{\ifthenelse{\equal{#4}{b}}{\fill (#2-0.5,#1+0.5-#3)circle(0.2);}{\ifthenelse{\equal{#4}{c}}{\draw (#2-0.5,#1+0.5-#3)[line width=1pt,gray!70]circle(0.2);}{}} }\ifthenelse{\equal{#4}{d}}{\draw (#2-0.5,#1+0.5-#3)[line width=1pt]circle (0.2);}{\ifthenelse{\equal{#4}{e}}{\draw[line width=1pt,gray!70](#2-0.7,#1+0.3-#3)--(#2-0.3,#1+0.3-#3)--(#2-0.5,#1+0.8-#3)--(#2-0.7,#1+0.3-#3);}{\ifthenelse{\equal{#4}{f}}{\draw[line width=1pt](#2-0.7,#1+0.3-#3)--(#2-0.3,#1+0.3-#3)--(#2-0.5,#1+0.8-#3)--(#2-0.7,#1+0.3-#3); }}}} %(#2-0.5,#1+0.5-#3)

\usetikzlibrary{patterns}

\newcounter{row}
\newcommand{\chh}[1][]{}

\newcommand{\Di}[1]{\mathrm{dim}\left(#1\right)} %prosexe an ta alaksis me claim2stotherem me convolutions

\newcommand{\X}[2][]{\mathbf{X}_{#1}(#2)}

\newcommand{\vsp}{\vspace{1,5mm}}
\newtheorem{theorem}{Theorem}
\newtheorem{expectation}[theorem]{Expectation}

\newtheorem{prop}[theorem]{Proposition}

\newtheorem{lem}[theorem]{Lemma}

\theoremstyle{definition}

\theoremstyle{remark}

\DeclareFontFamily{U}{mathb}{\hyphenchar\font45}
\DeclareFontShape{U}{mathb}{m}{n}{
      <5> <6> <7> <8> <9> <10> gen * mathb
      <10.95> mathb10 <12> <14.4> <17.28> <20.74> <24.88> mathb12
      }{}
\DeclareSymbolFont{mathb}{U}{mathb}{m}{n}

\DeclareMathSymbol{\act}{3}{mathb}{'374}
\DeclareMathSymbol{\ssquare}{3}{mathb}{"05}
\newtheoremstyle{subdefi}{1pt}{1pt}{}{}{\bfseries}{.}{.5em}{}
\theoremstyle{subdefi}
\newtheorem*{axbxxxt}{Subdefinition}  
   \newenvironment{subdef*}
     {\pushQED{\qed}\axbxxxt}
     {\popQED\endaxbxxxt}

\begin{document}
\title{A Note on a Fourier Expansion of Automorphic Functions that Nontrivially Involves the Euclidean Algorithm}
\date{}
\author{Eleftherios Tsiokos}
\maketitle
\begin{abstract}We extend the Fourier expansion of automorphic functions of Piateski-Shapiro and Shalika in a way that the Euclidean algorithm appears nontrivially. Then, we discuss a potential application in the Rankin-Selberg method.
 
\end{abstract}
Let $\kkk$ be a number field,  $\A$ be its adele ring, and $n$ be a positive integer. We define $GL_n(\A)$-automorphic forms and $GL_n(\A)$-automorphic representations in a standard way, exactly as we did in \cite{Tsiokos2} (in 9.2). Additionally, by a $P(\A)$-automorphic function, for a standard parabolic subgroup $P$ of $GL_n$ , we mean a function with domain $P(\kkk)\s P(\A) $ which is of uniform moderate growth, smooth, and $K\cap P(\A)$-finite.

By an additive function [always abbreviated by ``AF"], say $\F$, we mean an algebraic additive function on a unipotent algebraic group, and we denote by $D_\F$ this group. For $\gamma$ being a matrix in the general linear group to which we realize $D_\F$ we define by $\gamma\F$ to be the AF given by $\gamma\F(n)=\F(\gamma^{-1}n\gamma)$ for all $n\in D_{\gamma\F}=\gamma D_{\F}\gamma^{-1}$. The restriction of $\F$ in a group $L$ is denoted by $\F|_L$. We define $\psi_\kkk$ to be a unitary character of $\kkk\s \A$. If $\phi$ is an automorphic function with domain containing $D_\F$ we denote by $\F(\phi)$ the function with the same domain as $\phi$ given by $\F(\phi)(g):=\int_{D_\F(\kkk)\s D_\F(\A)}\phi(ng)\psi_\kkk^{-1}(\F(n))dn$.

The meaning of every symbol is retained until the end of the paper except if it is redefined.  

Let $n_1>n_2$ be two coprime positive integers. Let $s$ be the number of divisions in the Euclidean algorithm for the pair $(n_1,n_2)$ until we get remainder equal to $0$,  and for each $1\leq i\leq s$ let the $i$-th division be $n_i=n_{i+1}k_i+n_{i+2}$ where $k_i$ and $n_{i+2}$ are (the unique)  integers satisfying $k_i> 0$ and $n_{i+1}>n_{i+2}\geq 0$. We therefore have $n_{s+1}=1$ and $n_{s+2}=0$.

Let $P_1$ be the standard parabolic subgroup of $GL_{n_1+n_2}$ with Levi isomorphic to [and identified with] $ GL_{n_1}\times GL_{n_2}$ so that $GL_{n_1}$ appears in the upper left corner. Then for $1\leq i\leq k_1+...k_s$ we inductively define the $\kkk$-subgroups $V_i$ and $P_i$ of $GL_{n_1+n_2}$ and the AF $\F_{n_1,n_2}$ with domain $D_{\F_{n_1,n_2}}=\prod_{1\leq i\leq k_1+...+k_s}V_i$ by
requiring that:\begin{itemize}
	\item $V_i$ is the unipotent radical of $P_i$;
	\item $\F_{n_1,n_2}|_{V_i}$ is in the open orbit of the action by conjugation of $P_i$ on the AFs with domain $V_i$.
	\item $\{h\in P_i:h(\F_{n_1,n_2}|_{V_i})=\F_{n_1,n_2}|_{V_i}\}=P_{i+1}\rtimes V_i$. \end{itemize} 
\begin{prop}\label{Fex}Let $\phi$ be a $P_1(\A)$-automorphic function. The sum below is absolutely and uniformly (on $g$) convergent, and we denote by $f(\phi)$ the automorphic function on $P_1(\kkk)\s P_1(\A)$ it defines:  $$f(\phi)(g):=\sum_{\gamma\in(ZV)(\kkk)\s (GL_{n_1}\times  GL_{n_2})(\kkk)}\F_{n_1,n_2}(\phi)(\gamma g) \qquad\forall g\in P_1(\A)$$where $Z$ denotes the center of $GL_{n_1+n_2}$.\end{prop}
\begin{proof}We proceed inductively on $n_1+n_2$. The Fourier expansion $\phi$ over $V_1$ takes the form $$\phi(g)=\left(\sum_{\gamma\in P_2(\kkk)\s (GL_{n_1}\times GL_{n_2})(\kkk)}\F_{n_1,n_2}|_{V_1}(\phi)(\gamma g)\right)+\text{other terms}. $$Since $\F_{n_1,n_2}|_{V_1}(\phi) $ is addressed by the inductive hypothesis, we are done.
\end{proof}
\begin{expectation}\label{ex}Let $\phi^\prime$ and $\phi$ respectivley be a $GL_{n_1}(\A)$-automorphic function and a $P_1(\A)$-automorphic function. Then for any complex number $t$ with sufficiently big real part, the two integrals below: are absolutely convergent, they are equal, and we denote them by $I$: \begin{multline}I:=\label{pmult}\int_{(GL_{n_1}\times GL_{n_2})(\kkk)Z(\A)\s (GL_{n_1}\times GL_{n_2})(\A)}\phi^\prime(g_1)f(\phi)\begin{pmatrix}g_1&\\&g_2\end{pmatrix}|\det(g_1)|^{n_2t}|\det(g_2)|^{-n_1t}dg_1dg_2=\\\int_{(V^1\times Z))(\A)\s GL_{n_1}(\A)\times GL_{n_2}(\A)} \F_{n_1-n_2,n_2}(\phi^\prime)(g_1)\F_{n_1,n_2}(\phi)\begin{pmatrix}g_1&\\&g_2\end{pmatrix}|\det(g_1)|^{n_2t}|\det(g_2)|^{-n_1t}dg_1dg_2\end{multline}where $V^1=\prod_{2\leq i\leq k_1+...+k_s}V_i,$ and $\F_{n_1-n_2,n_2}$ is defined in the same way as $\F_{n_1,n_2}$ by replacing $n_1$ with $n_1-n_2$ except if $k_1=1$, then we replace: $n_1$ with $n_2$, $n_2$ with $n_1-n_2$, and ``upper left corner" with ``lower right corner".
	\end{expectation} Note that equality (\ref{pmult}) is ``formally obtained" directly from the definition of $f(\phi)$. We fully skip in the present note to discuss (even with expectations) about analytic continuation in $t$ or in any other variable. 

Since the domain of integration in the second expression of $I$ is factorizable, to traditionally (in contrast to ``new way" integrals) search for choices of data that make $I$ factorizable and nonzero we ask: which choices of data make the two Fourier coefficients involved factorizable and nonzero. And of course---after letting all the data vary (in particular $n_1,n_2$)---this is the same as asking: for which $\phi$ the function $\F_{n_1,n_2}(\phi)$ is factorizable and nonzero.

We make use below of the concepts $\Omm{\F}$ and $\BBnk[n]$ which are defined in \cite{Tsiokos2} respectively in Definitions 8.1.1 and 6.8.  Definition 8.1.1 states $\Omm{\F}:=\{a\in\Om(\F):\mult{a}{\F}<\infty\}$ where $\Om(\F)$ consists of the ``minimal orbits attached" to $\F$  and $\mult{a}{\F}$ is a form of multiplicity of $a$ in $\F$. Through an easy refinement (on the ``only if" in (ii)) of Corollary 5.6 in \cite{Tsiokos2} (which is explained in Lemma \ref{lemco} below), for every $GL_n(\A)$-automorphic representation $\pi$ with (a) its elements admitting an (of course unique) absolutely convergent Eisenstein series expansion over discrete data and (b) the maximal orbit attached to it belonging in $\Omm{\F}$, we express $\F(\pi)$ as a finite sum (with $\mult{a}{\F}$ terms) of factorizable functions\footnote{In case $\F\in\BBnk[n]$ we have by Main corollary 6.17 that $\Omm{\F}=\{a\in\Om(\F):\mult{a}{\F}=1\}$; therefore the sum consist of only one term (and without modifying Corollary 5.6).}. The set $\BBnk[n]$ consists of AFs $\F$ for which a way is found in Main Corollary 6.17 in \cite{Tsiokos2} to shorten the calculation of $\Omm{\F}$. By closely following the construction of the Fourier expansion of the present note we directly obtain from the definition of $\BBnk[n_1+n_2] $ that $\F_{n_1,n_2}\in\BBnk[n_1+n_2]$. The set $\Omm{\F_{n_1,n_2}}$ has been calculated in the special case $s=2$ (recall this means that $n_1$ leaves residue $1$ when divided by $n_2$) in Theorem 8.3.12 in \cite{Tsiokos2}. Back to the general case, there is always at least one element:\begin{lem}\begin{equation}\label{orbit}[(k_1+...+k_s+1)^{n_{s+1}},(k_1+...+k_{s-1})^{n_s-n_{s+1}},...,k_1^{n_2-n_3}]\in\Omm{\F_{n_1,n_2}}.\end{equation}
\end{lem}\begin{tproof}Some basic results on algebraic groups and on nilpotent orbits are used freely. The Lie algebra of any algebraic group $H$ is denoted by $\Lie{H}.$ Notice that the orbit in the left hand side of (\ref{orbit}) is the Richardson orbit openly intersecting the unipotent radicals of the  $GL_{n_1+n_2}$-parabolics with Levi isomorphic to $GL_{n_{s+1}}^{k_s+1}\times GL_{n_s}^{k_{s-1}}\times... GL_{n_2}^{k_1}$. For any vector subspace $X$ of $\Lie{GL_{n_1+n_2}}$ we denote by $\Pr_{X}$ the projection (with respect to $(X,Y\rightarrow\mathrm{tr}(XY^t))$) of $\Lie{GL_{n_1+n_2}}$ to $X$. Without loss of generality, we choose $\F_{n_1,n_2}$ so that: for each matrix entry there is at most one $V_i$ containing matrices nontrivial in this entry; for every $1\leq i\leq k_1+...+k_s$ the matrix $\Pr_{\Lie{V_i}}(J_{\F_{n_1,n_2}})$ is  upper triangular and a Weyl group conjugate of a Jordan matrix. Let $J_i$
 be the matrix obtained from $\Pr_{\Lie{V_i}}(J_{\F_{n_1,n_2}})$  by replacing with zero  all the nonzero nondiagonal entries except the one among them lying in the smallest possible row. We define $J:=J_1+...J_{k_1+...k_s}$. Let $\Om(J)$ be the nilpotent orbit containing $J$. The variety $\X{...}$ is defined for example in the beginning of Section 6 in \cite{Tsiokos2}. Notice that: $J\in\X{\F_{n_1,n_2}}$, and hence by the ``$\X{...}$-version"\footnote{That is, the equivalent (due to Lemma 6.6) statement to Main corollary 6.17 which is obtained from it by replacing the second sentence with ``let $a\in\UUU{n}$ be such that $\X{\F}\cap a\not=\emptyset$". } of Main Corollary 6.17 in \cite{Tsiokos2} we obtain $\Di{\Om(J)}\geq2\Di{D_{\F_{n_1,n_2}}}$; for $V$ being a unipotent radical as in the first sentence of the proof for an appropriate order of the blocks we have $J\in V$, and hence $\Di{\Om(J)}\leq2\Di{V}$. By also noting $\Di{D_{\F_{n_1,n_2}}}=\Di{V}$, we obtain $\Di{\Om(J)}=2\Di{_{\F_{n_1,n_2}}}$ and hence we are done by Main corollary 6.17 in \cite{Tsiokos2}.
 \end{tproof} 

As in familiar integrals that are special cases of $I$, we can extend the topic of the previous paragraph by choosing $\phi$ among Fourier coefficient of automorphic functions on a  reductive group; we restrict our attention to the case that this reductive group is a general linear group, say $GL_N$, where $GL_{n_1+n_2} $ is identified as needed with a copy it admits inside $GL_N$.  Let $\Z$ be the AF corresponding\footnote{That is, $\phi=\Z(\phi^{\prime\prime})$ for a $GL_N(\A)$-automorphic function $\phi^{\prime\prime}$.} to a Fourier coefficient expression of a choice of  $\phi$, and let $\circ$ be the operation on AFs corresponding to composition of Fourier coefficients as it is precisely given in Definition 2.5.1 in \cite{Tsiokos2}. Then we need to study $\Omm{\F_{n_1,n_2}\circ\Z}$, and for many choices of $\Z$ satisfying $\F_{n_1,n_2}\circ\Z\in\BBnk[N]$,  Main corollary 6.17 in \cite{Tsiokos2} should again be useful.
\subsection*{Appendix}Below we adopt and use without mention notations from \cite{Tsiokos2}.\begin{lem}\label{lemco} Let $\pi$ and $\F\in\AAA[,n]$ and $\pi\in\Aut{\kkk,n,>}$. Assume that $\Om^\prime(\pi)\in\Om(\F)$ and that an $(\F\rightarrow\Pri[,n],GL_n,\kkk)$-tree has only finitely many, say $k$, output vertices with label in $\Pri[,n][\Om^\prime(\pi)]$. Then there is an $\phi\in\pi$ such that $\F(\phi)$ is the sum of $k$ factorizable functions.
\end{lem}\begin{tproof}The proof is obtained directly from the proof of Corollary 5.6 in \cite{Tsiokos2} and the following claim:

\vsp 
\noindent\textbf{Claim.}\textit{ Let $\F\in\AAA[,n]$. Consider operations $\exchange{Y}{X}$ and $\e(V)$ (defined with the current choices of $\kkk$ and $n$), such that $\e(V)(\exchange{Y}{X}(\F))$ is defined. Then, for any AF $\QQ\in\e(V)(\exchange{Y}{X}(\F))$ the integral \begin{equation}\label{rint}\int_{(Y\cap C)(\A)\s Y(\A)}\QQ(\phi)(n)dn \end{equation} is absolutely convergent, where $C$ is as in Definition 2.5.11 in \cite{Tsiokos2} (for the current choices of $X,Y,\F$).}

\vsp
\noindent\textit{Proof of Claim. }$\F(\phi)(1)=\int_{Y(\kkk)(Y\cap C)(\A)\s Y(\A)}\F|_C(\phi)(y)dy. $ We apply to $F|_C(\phi)$ the Fourier expansion over $X(\kkk)(X\cap C)(\A)\s X(\A)$, and then to each term $\Z$ we apply the Fourier expansion over $vVv^{-1}(\kkk)(v Vv^{-1}\cap XC)(\A)\s vVv^{-1}(\A)$ where $v$ is the element in $Y(\kkk)$ satisfying $v^{-1}\Z =\exchange{X}{Y}(\F)$. Among the terms obtained consider the ones of the form $v\QQ(\phi)$ for any $v$ as previously; then we integrate over $\int_{Y(\kkk)(Y\cap C)(\A)\s Y(\A)}$ the sum of these terms, and by using Fubini's theorem we obtain the absolute convergence of the integral in (\ref{rint}).$\hspace*{134pt}\square$Claim
\end{tproof}


\begin{thebibliography}{88}
		\bibitem[T]{Tsiokos2}E. Tsiokos. On Fourier Coefficients of $GL(n)$-Automorphic Functions over Number Fields {\tt arXiv:1711.11545v1}.
\end{thebibliography}
\end{document}